\documentclass{amsart}

\usepackage{amsfonts, amssymb, amsmath, eucal, verbatim, amsthm, amscd, enumerate}

\newtheorem{theorem}{Theorem}[section]

\newtheorem{corollary}[theorem]{Corollary}
\theoremstyle{definition}

\theoremstyle{remark}

\newtheorem*{remarks}{Remarks}

\numberwithin{equation}{section}
\begin{document}

\keywords{Trace estimates, extremisers}

\author{Neal Bez}
\address{Neal Bez and Shuji Machihara, Department of Mathematics, Graduate School of Science and Engineering,
Saitama University, Saitama 338-8570, Japan}
\email{\{nealbez,matihara\}@mail.saitama-u.ac.jp}
\author{Shuji Machihara}
\author{Mitsuru Sugimoto}
\address{Mitsuru Sugimoto, Graduate School of Mathematics, Nagoya University \\
Furocho, Chikusa-ku, Nagoya 464-8602, Japan}
\email{sugimoto@math.nagoya-u.ac.jp}

\title{Extremisers for the trace theorem on the sphere}

\begin{abstract}
We find all extremisers for the trace theorem on the sphere. We also provide a sharp extension for functions belonging to certain Sobolev spaces with angular regularity.
\end{abstract}

\maketitle

\section{Introduction}
Suppose $d \geq 2$, $s \in (\frac{1}{2},\frac{d}{2})$ and let $\mathcal{R}$ denote the operator which restricts complex-valued functions on $\mathbb{R}^d$ to the unit sphere $\mathbb{S}^{d-1}$ in $\mathbb{R}^d$. The classical trace theorem on $\mathbb{S}^{d-1}$ asserts that $\mathcal{R}$ is a bounded linear operator from the homogeneous Sobolev space of order $s$, denoted $\dot{H}^s(\mathbb{R}^d)$, to $L^2(\mathbb{S}^{d-1},\mathrm{d}\sigma)$, where $\mathrm{d}\sigma$ denotes the induced Lebesgue measure on $\mathbb{S}^{d-1}$. The exact value of the operator norm was found very recently by Ruzhansky and Sugimoto in \cite{RStrace}; in particular, they showed that 
\begin{equation} \label{e:RStrace}
\| \mathcal{R} f \|_{L^2(\mathbb{S}^{d-1},\mathrm{d}\sigma)}^2 \leq 2^{1-2s} \frac{\Gamma(2s-1)\Gamma(\frac{d}{2}-s)}{\Gamma(s)^2\Gamma(\frac{d}{2} - 1 + s)} \| f\|_{\dot{H}^s(\mathbb{R}^d)}^2
\end{equation}
holds for all $f$ in $\dot{H}^s(\mathbb{R}^d)$, and that the constant in \eqref{e:RStrace} is optimal. 

In \cite{RStrace} the question of which (if any) functions are extremisers for \eqref{e:RStrace} (i.e. nonzero functions $f \in \dot{H}^s(\mathbb{R}^d)$ such that we have equality in \eqref{e:RStrace}) was not answered. In this paper we answer this question by establishing the following characterisation of extremisers.
\begin{theorem} \label{t:extremisers}
Let $d \geq 2$ and $s \in (\frac{1}{2},\frac{d}{2})$. Then $f$ is an extremiser for \eqref{e:RStrace} if and only if
\[
f \in \emph{span}\left(\frac{1}{|\cdot|^{d-2s}} * \mathrm{d}\sigma \right) \setminus \{0\}\,,
\]
or equivalently, the Fourier transform, $\widehat{f}$, of $f$ is such that
\[
\widehat{f} \in \emph{span}\left(\frac{J_{\frac{d}{2}-1}(|\cdot|)}{|\cdot|^{\frac{d}{2} + 2s - 1}}  \right) \setminus \{0\}\,.
\]
\end{theorem}
Here, $*$ denotes convolution and $J_{\frac{d}{2}-1}$ is the Bessel function of the first kind of order $\frac{d}{2}-1$. When $d=3$ the function $|\cdot|^{2s-d} * \mathrm{d}\sigma$ has a particularly nice explicit expression. Indeed, for $s \in (\frac{1}{2},\frac{3}{2})$, one can use the rotation invariance of $\mathrm{d}\sigma$ to obtain
\[
\frac{1}{|\cdot|^{3-2s}} * \mathrm{d}\sigma (x) = C(s) \, \frac{(|x|+1)^{2s-1} - \big||x| - 1\big|^{2s-1}}{|x|} 
\]
for some nonzero constant $C(s)$. Specialising further to the case $s=1$ we see that extremisers to the optimal trace estimate
\[
\| \mathcal{R} f\|_{L^2(\mathbb{S}^2,\mathrm{d}\sigma)} \leq \|\nabla f\|_{L^2(\mathbb{R}^3)}
\] 
are precisely nonzero multiples of $f$ given by
\begin{equation*}
f(x) = \left\{\begin{array}{cccc} 1 & \text{if $|x| \leq 1$} \\ \frac{1}{|x|} & \text{if $|x| > 1$} \end{array} \right.\,.
\end{equation*}

Since we will be handling explicit constants, we clarify that our convention for the Fourier transform is
\[
\widehat{f}(\xi) = \int_{\mathbb{R}^d} f(x) \exp(-i x \cdot \xi) \, \mathrm{d}x\,,
\]
so that the inverse is given by
\[
f^{\vee}(x) = \frac{1}{(2\pi)^d} \int_{\mathbb{R}^d} \widehat{f}(\xi) \exp(i x \cdot \xi) \, \mathrm{d}\xi\,,
\]
for appropriate functions $f : \mathbb{R}^d \to \mathbb{C}$.

The main ingredient in proving Theorem \ref{t:extremisers} is a duality argument and an analysis of weighted $L^2$-estimates for the Fourier extension operator $G \mapsto \widehat{G\mathrm{d}\sigma}$ associated to the sphere, where the weight is homogeneous; see the forthcoming Theorem \ref{t:S1}. With this approach we prove rather more than Theorem \ref{t:extremisers}; we present a sharp version of recent dual trace estimates due to Fang and Wang \cite{FangWang} which incorporate certain angular regularity and permit reverse estimates too. This allows us to extend the sharp estimate in \eqref{e:RStrace} to certain sharp trace estimates for $f$ belonging to Sobolev spaces which incorporate angular regularity and we characterise the extremisers.

In addition, we find the exact operator norm of $\mathcal{R}$ as a mapping $\dot{H}^s(\mathbb{R}^d) \to L^p(\mathbb{S}^{d-1},\mathrm{d}\sigma)$ for the same range of $d$ and $s$, where $p = \frac{2(d-1)}{d-2s}$. This is, in fact, a stronger estimate than \eqref{e:RStrace} because $p > 2$ and we may apply H\"older's inequality on $\mathbb{S}^{d-1}$ to deduce \eqref{e:RStrace}. There is no loss of optimality in the constant under this application of H\"older's inequality because the class of extremisers for $\mathcal{R} : \dot{H}^s(\mathbb{R}^d) \to L^p(\mathbb{S}^{d-1},\mathrm{d}\sigma)$ contains certain radial functions, which means their restriction to the sphere is constant. We obtain the operator norm of $\mathcal{R} : \dot{H}^s(\mathbb{R}^d) \to L^p(\mathbb{S}^{d-1},\mathrm{d}\sigma)$ via the sharp Hardy--Littlewood--Sobolev inequality on $\mathbb{S}^{d-1}$, due to Lieb \cite{Lieb}. We shall use Lieb's characterisation of extremisers for this inequality to obtain a complete description of the extremisers for $\mathcal{R} : \dot{H}^s(\mathbb{R}^d) \to L^p(\mathbb{S}^{d-1},\mathrm{d}\sigma)$. This class of extremisers turns out to be larger than the class  of extremisers in Theorem \ref{t:extremisers}; see the forthcoming Theorem  \ref{t:sharpsharptrace} (this result was independently obtained by Beckner in \cite{Becknermulti}.).


\section{Proof of Theorem \ref{t:extremisers}} \label{section:theta1}
Before proceeding with the full proof of Theorem \ref{t:extremisers}, we enter into a few preliminary remarks to explain how the class of extremisers can be seen to arise. It is natural to expect that certain radial functions are amongst the class of extremisers for \eqref{e:RStrace}. The Euler--Lagrange equation for \eqref{e:RStrace} is
\begin{equation} \label{e:EL}
\widehat{\mathcal{R}^*\mathcal{R} f} (\xi) = \lambda(f)  \widehat{f}(\xi) |\xi|^{2s} 
\end{equation}
for almost all $\xi \in \mathbb{R}^d$ and for some constant $\lambda(f)$ which depends on $f$. A calculation shows that if $f$ is radial then $\widehat{\mathcal{R}^*\mathcal{R} f} $ is a constant multiple of $\widehat{\mathrm{d}\sigma}$ and \eqref{e:EL} implies that $\widehat{f}(\xi)$ is a constant multiple of
\[
\frac{\widehat{\mathrm{d}\sigma}(\xi)}{|\xi|^{2s}} = (2\pi)^{\frac{d}{2}}
\frac{J_{\frac{d}{2}-1}(|\xi|)}{|\xi|^{\frac{d}{2}+2s-1}}\,.
\]
This argument shows that if radial extremisers exist then they necessarily have this form, and one may show that such $f$ are indeed extremisers by substituting into both sides of \eqref{e:RStrace} and calculating everything explicitly. Showing the uniqueness of extremisers requires more work, and this will be established in the remainder of this section.

We shall abbreviate $L^2(\mathbb{S}^{d-1},\mathrm{d}\sigma)$ to $L^2(\mathbb{S}^{d-1})$ and we use the decomposition 
\[
L^2(\mathbb{S}^{d-1}) = \bigoplus_{k=0}^\infty \mathcal{H}_k\,,
\]
where $\mathcal{H}_k$ denotes the space of spherical harmonics of degree $k$. Let $(\lambda_k)_{k \in \mathbb{N}_0}$ be the sequence given by
\[
\lambda_k = 2^{1-2s} \frac{\Gamma(2s-1)\Gamma(k+\frac{d}{2}-s)}{\Gamma(s)^2\Gamma(k+\frac{d}{2} - 1 + s)}\,.
\]
Also, let $\mathcal{S} : L^2(\mathbb{R}^d) \to L^2(\mathbb{S}^{d-1})$ be the linear operator given by
\[
\mathcal{S} g= \mathcal{R} D^{-s}g \,,
\]
where $D = \sqrt{-\Delta}$. The adjoint operator $\mathcal{S}^* : L^2(\mathbb{S}^{d-1}) \to L^2(\mathbb{R}^d)$ can be calculated as
\[
\mathcal{S}^* G = (|\cdot|^{-s}\widehat{G \mathrm{d}\sigma})^{\vee}\,.
\]
Key to our proof of Theorem \ref{t:extremisers} is the following spectral decomposition of $\mathcal{S}\mathcal{S}^*$.
\begin{theorem} \label{t:S1}
Let $d \geq 2$ and $s \in (\frac{1}{2},\frac{d}{2})$. If $k \in \mathbb{N}_0$ and $P_k$ is any spherical harmonic of degree $k$, then
\begin{equation} \label{e:S1eigen}
\mathcal{S}\mathcal{S}^* P_k = \lambda_k P_k\,.
\end{equation}
The sequence of eigenvalues $(\lambda_k)_{k \in \mathbb{N}_0}$ is a strictly decreasing sequence converging to zero and consequently
\begin{equation} \label{e:S*bound}
\| \mathcal{S}^*G \|_{L^2(\mathbb{R}^d)}^2 \leq   \lambda_0 \|G\|_{L^2(\mathbb{S}^{d-1})}^2\,,
\end{equation}
where the constant is optimal and equality holds if and only if $G$ is constant.
\end{theorem}
\begin{proof}[Proof of Theorem \ref{t:S1}]
First observe that, for $\omega \in \mathbb{S}^{d-1}$ we have
\begin{equation} \label{e:SS*}
\mathcal{S}\mathcal{S}^*G(\omega) = 2^{-2s} \pi^{-\frac{d}{2}} \, \frac{\Gamma(\frac{d}{2}-s)}{\Gamma(s)} \int_{\mathbb{S}^{d-1}} \frac{G(\varphi)}{|\omega - \varphi|^{d-2s}} \, \mathrm{d}\sigma(\varphi)\,,
\end{equation}
where we have used Fourier inversion and the formula
$$
\widehat{\frac{1}{|\cdot|^{d-\zeta}}}(\xi) = 2^\zeta \pi^{\frac{d}{2}} \frac{\Gamma(\frac{1}{2}\zeta)}{\Gamma(\frac{1}{2}(d-\zeta))} \,\frac{1}{|\xi|^\zeta}\,,
$$
for $\zeta \in (0,d)$, giving the Fourier transform of a Riesz potential. If $k \in \mathbb{N}_0$ and $P_k$ is a spherical harmonic of degree $k$, then
it follows immediately from \eqref{e:SS*} and Lemma 5.1 from \cite{BS} that $\mathcal{S}\mathcal{S}^*P_k = \lambda_kP_k$. Lemma 5.1 of \cite{BS} also says that the sequence of eigenvalues $(\lambda_k)_{k \in \mathbb{N}_0}$ is strictly decreasing to zero.

We may now write
\[
\mathcal{S} \mathcal{S}^* = \sum_{k \in \mathbb{N}_0} \lambda_k H_k\,,
\]
where $H_k$ is the projection operator $L^2(\mathbb{S}^{d-1}) \to \mathcal{H}_k$. This means that for each $G \in L^2(\mathbb{S}^{d-1})$ we have
\begin{equation} \label{e:expansion}
\| \mathcal{S}^* G \|_{L^2(\mathbb{R}^d)}^2 = \sum_{k \in \mathbb{N}_0} \lambda_k \|H_k G\|_{L^2(\mathbb{S}^{d-1})}^2
\end{equation}
and \eqref{e:S*bound} clearly follows, with equality when $G \in \mathcal{H}_0$; that is, $G$ is constant. There are no further cases of equality because if $G$ is an extremiser then 
\begin{equation*} 
\sum_{k \in \mathbb{N}_0} \lambda_k \|H_k G\|_{L^2(\mathbb{S}^{d-1})}^2 = \| \mathcal{S}^*G \|_{L^2(\mathbb{R}^d)}^2 = \lambda_0 \|G\|_{L^2(\mathbb{S}^{d-1})}^2 = \sum_{k \in \mathbb{N}_0} \lambda_0 \|H_k G\|_{L^2(\mathbb{S}^{d-1})}^2
\end{equation*}
and the strict decreasingness of $(\lambda_k)_{k \in \mathbb{N}_0}$ forces $H_k G = 0$ for $k \neq 0$. This completes our proof of Theorem \ref{t:S1}.
\end{proof}
By duality, \eqref{e:RStrace} is an immediate consequence of \eqref{e:S*bound}. We now show how to use the characterisation of extremisers in Theorem \ref{t:S1} to deduce Theorem \ref{t:extremisers}. 
\begin{proof}[Proof of Theorem \ref{t:extremisers}]
Writing $f = D^{-s}g$, we have that $f$ is an extremiser for \eqref{e:RStrace} if and only if
\begin{equation} \label{e:Sdualex}
\| \mathcal{S}g \|_{L^2(\mathbb{S}^{d-1})}^2 = \lambda_0 \|g\|_{L^2(\mathbb{R}^d)}^2\,.
\end{equation}
Also, $f$ belongs to the span of $|\cdot|^{2s-d} * \mathrm{d}\sigma$ if and only if $g$ belongs to the span of $|\cdot|^{s-d} * \mathrm{d}\sigma$, or equivalently, that $g$ belongs to the span of $\mathcal{S}^*\mathbf{1}$, where $\mathbf{1}$ denotes the constant function taking the value 1 everywhere. Thus, it suffices to show that \eqref{e:Sdualex} holds if and only if $g \in \mathcal{S}^*(\mathcal{H}_0)$. 

Firstly, if $g \in L^2(\mathbb{R}^d) \setminus \{0\}$ is such that \eqref{e:Sdualex} holds, and if $G \in L^2(\mathbb{S}^{d-1})$ is given by
\[
G = \frac{\mathcal{S} g}{\| \mathcal{S} g \|_{L^2(\mathbb{S}^{d-1})}}\,,
\]
then we have 
\begin{align*}
\lambda_0\| g\|_{L^2(\mathbb{R}^d)}^2 & = \| \mathcal{S} g \|^2_{L^2(\mathbb{S}^{d-1})}  \\
& = \langle \mathcal{S} g, G \rangle_{L^2(\mathbb{S}^{d-1})}^2 \\
& = \langle g, \mathcal{S}^* G \rangle_{L^2(\mathbb{R}^{d})}^2 \leq \|g\|_{L^2(\mathbb{R}^{d})}^2 \|\mathcal{S}^*G\|_{L^2(\mathbb{R}^{d})}^2 \leq   \lambda_0 \|g\|_{L^2(\mathbb{R}^{d})}^2 \,.
\end{align*}
This means $G$ is an extremiser for \eqref{e:S*bound} and hence $G \in \mathcal{H}_0$. Also, from equality in the above application of the Cauchy--Schwarz inequality we know $g$ and $\mathcal{S}^*G$ are linearly dependent, which means $g \in \mathcal{S}^*(\mathcal{H}_0)$, as desired. Conversely, if $g = \mathcal{S}^*G$ for some $G \in \mathcal{H}_0$, then from \eqref{e:S1eigen} we have $\mathcal{S} g = \lambda_0 G$, which implies
\[
\| \mathcal{S} g \|^2_{L^2(\mathbb{S}^{d-1})} = \lambda_0 \langle \mathcal{S} g, G \rangle_{L^2(\mathbb{S}^{d-1})} = \lambda_0\langle g, \mathcal{S}^* G \rangle_{L^2(\mathbb{R}^{d})} = \lambda_0 \|g\|_{L^2(\mathbb{R}^d)}^2 \,.
\]
Therefore $g$ satisfies \eqref{e:Sdualex}, as desired.
\end{proof}

\section{Sharp trace theorems with angular regularity} \label{section:angular}
We may easily generalise the arguments in the previous section to obtain certain sharp trace estimates which allow the inclusion of angular regularity. To state our results in this direction, it is necessary to establish some further notation. 

We write $-\Lambda$ for the Laplace--Beltrami operator on $\mathbb{S}^{d-1}$ and use the well-known fact that the eigenvalues of $-\Lambda$ are $k(k+d-2)$ for $k \in \mathbb{N}_0$ and the corresponding eigenspaces are $\mathcal{H}_k$. Thus 
\[
-\Lambda = \sum_{k=0}^\infty k(k+d-2)H_k\,,
\] 
where, as before, we use $H_k$ for the projection from $L^2(\mathbb{S}^{d-1})$ to $\mathcal{H}_k$. The operator $H_k$ may be written as
\[
H_kG(\omega) = \frac{N_{k,d}}{|\mathbb{S}^{d-1}|} \int_{\mathbb{S}^{d-1}} G(\varphi) P_{k,d}(\omega \cdot \varphi) \, \mathrm{d}\sigma(\varphi)\,,
\]
where $P_{k,d}$ is the Legendre polynomial of degree $k$ in $d$ dimensions and 
\[
N_{k,d} = \frac{(2k+d-2)(k+d-3)!}{k!(d-2)!}\,.
\]
Thus, we may homogeneously extend $H_k$ to functions on $\mathbb{R}^d$ by setting
\[
H_kg(x) = \frac{N_{k,d}}{|\mathbb{S}^{d-1}|} \int_{\mathbb{S}^{d-1}} g(|x|\varphi) P_{k,d}(\tfrac{x}{|x|} \cdot \varphi) \, \mathrm{d}\sigma(\varphi)
\]
and, furthermore, we may define 
\[
\theta(-\Lambda) = \sum_{k=0}^\infty \theta(k(k+d-2))H_k
\]
on either $\mathbb{S}^{d-1}$ or $\mathbb{R}^d$, for functions $\theta$ on $[0,\infty)$.

For appropriate functions $\theta$, let $\mathcal{S}_\theta : L^2(\mathbb{R}^d) \to L^2(\mathbb{S}^{d-1})$ be the linear operator given by
\[
\mathcal{S}_\theta g= \mathcal{R}  \,\overline{\theta}(-\Lambda)  D^{-s}g 
\]
with adjoint operator $\mathcal{S}_\theta^* : L^2(\mathbb{S}^{d-1}) \to L^2(\mathbb{R}^d)$ given by
\[
\mathcal{S}_\theta^* G = (|\cdot|^{-s} \theta(-\Lambda) \widehat{G \mathrm{d}\sigma})^{\vee}\,.
\]
Of course, the operator $\mathcal{S}$ from Section \ref{section:theta1} coincides with $\mathcal{S}_\mathbf{1}$.

We also introduce the notation $\lambda_k(\theta)$ for the sequence given by
\[
\lambda_k(\theta) = 2^{1-2s} \frac{\Gamma(2s-1)\Gamma(k+\frac{d}{2}-s)}{\Gamma(s)^2\Gamma(k+\frac{d}{2} - 1 + s)}|\theta(k(k+d-2))|^2
\]
for appropriate functions $\theta$ on $[0,\infty)$, and related index sets $\mathbf{k}$ and $\mathbf{K}$ by 
\[
\mathbf{k} = \{ k \in \mathbb{N}_0 : \lambda_k(\theta) = \inf_{\ell \in \mathbb{N}_0} \lambda_\ell(\theta)\} \quad \text{and} \quad \mathbf{K} = \{ k \in \mathbb{N}_0 : \lambda_k(\theta) = \sup_{\ell \in \mathbb{N}_0} \lambda_\ell(\theta)\}\,.
\]
The following substantially generalises Theorem \ref{t:S1}.
\begin{theorem} \label{t:Stheta}
Let $d \geq 2$ and $s \in (\frac{1}{2},\frac{d}{2})$. If $k \in \mathbb{N}_0$ and $P_k$ is any spherical harmonic of degree $k$, then
\begin{equation} \label{e:Sthetaeigen}
\mathcal{S}_\theta\mathcal{S}_\theta^* P_k = \lambda_k(\theta)  P_k\,.
\end{equation}
Consequently, we have
\begin{equation} \label{e:Sthetamain} 
\inf_{k \in \mathbb{N}_0} \lambda_k(\theta) \|G\|_{L^2(\mathbb{S}^{d-1})}^2 \leq \| \mathcal{S}_\theta^*G \|_{L^2(\mathbb{R}^d)}^2 \leq  \sup_{k \in \mathbb{N}_0} \lambda_k(\theta) \|G\|_{L^2(\mathbb{S}^{d-1})}^2\,,
\end{equation}
where the constants are optimal. If $\inf_{k} \lambda_k(\theta) > 0$, then the extremisers for the lower bound are precisely the nonzero elements of $\bigoplus_{k \in \mathbf{k}} \mathcal{H}_k$, and if $\sup_{k} \lambda_k(\theta) < \infty$, then the extremisers for the upper bound are precisely nonzero elements of $\bigoplus_{k \in \mathbf{K}} \mathcal{H}_k$. 
\end{theorem}
\begin{remarks}
(1) The above statement should be interpreted appropriately in the sense that if the index set $\mathbf{k}$ (respectively, $\mathbf{K}$) is empty, then the lower bound (respectively, upper bound) in \eqref{e:Sthetamain} has no extremisers.

(2) It is an easy consequence of Stirling's approximation that, for some constants $c(d,s) > 0$ and $C(d,s) < \infty$ we have
\[
\frac{c(d,s)}{(1+k)^{2s-1}}  \leq 2^{1-2s} \frac{\Gamma(2s-1)\Gamma(k+\frac{d}{2}-s)}{\Gamma(s)^2\Gamma(k+\frac{d}{2} - 1 + s)} \leq \frac{C(d,s)}{(1+k)^{2s-1}} 
\]
for all $k \in \mathbb{N}_0$. From this, for a given $\theta$, one should expect to be able to easily verify the hypotheses that $\inf_{k} \lambda_k(\theta) > 0$ and $\sup_{k} \lambda_k(\theta) < \infty$ in Theorem \ref{t:Stheta}.
\end{remarks}
\begin{proof}[Proof of Theorem \ref{t:Stheta}]
Since $\mathcal{S}_\theta = \overline{\theta}(-\Lambda) \mathcal{S}_\mathbf{1}$, we obtain immediately from \eqref{e:S1eigen} that $\mathcal{S}_\theta\mathcal{S}_\theta^*P_k = \lambda_k(\theta) P_k$. It follows that
\begin{equation} \label{e:expansion}
\| \mathcal{S}_\theta^* G \|_{L^2(\mathbb{R}^d)}^2 = \sum_{k \in \mathbb{N}_0} \lambda_k(\theta) \|H_k G\|_{L^2(\mathbb{S}^{d-1})}^2
\end{equation}
and the bounds \eqref{e:Sthetamain} clearly follow. We also note here that if $\mathbf{K} = \emptyset$ (respectively $\mathbf{k} = \emptyset$) then the upper bound (respectively, the lower bound) in \eqref{e:Sthetamain} holds strictly for $G \in L^2(\mathbb{S}^{d-1}) \setminus \{0\}$, and hence there are no extremisers. Despite this, it is easy to see that the bounds are optimal in this case.

In the case where $\mathbf{K}$ is nonempty, suppose we have equality in the upper bound, so that
\begin{equation} \label{e:upperex}
\| \mathcal{S}_\theta^*G \|_{L^2(\mathbb{R}^d)}^2 = \lambda_k(\theta) \|G\|_{L^2(\mathbb{S}^{d-1})}^2 
\end{equation}
for any choice of $k \in \mathbf{K}$. Then
\[
\sum_{\ell \in \mathbb{N}_0} \lambda_\ell(\theta) \|H_\ell G\|_{L^2(\mathbb{S}^{d-1})}^2  =  \sum_{\ell \in \mathbb{N}_0} \lambda_k(\theta) \|H_\ell G\|_{L^2(\mathbb{S}^{d-1})}^2  \,.
\]
Since $\lambda_\ell(\theta) < \lambda_k(\theta)$ whenever $\ell \notin \mathbf{K}$, we have $H_\ell G = 0$ for $\ell \notin \mathbf{K}$, and hence $G = \sum_{\ell \in \mathbf{K}} H_\ell G \in (\bigoplus_{\ell \in \mathbf{K}} \mathcal{H}_\ell) \setminus \{0\}$ as desired. Conversely, it is clear that if $G = \sum_{\ell \in \mathbf{K}} H_\ell G$ then \eqref{e:upperex} holds. A similar argument shows that the space of extremisers for the lower bound in \eqref{e:Sthetamain} is precisely $(\bigoplus_{\ell \in \mathbf{k}} \mathcal{H}_\ell) \setminus \{0\}$.
\end{proof}

\subsection*{Sharp Fourier extension equivalences with angular regularity}
Via an application of Plancherel's Theorem, as an aside, we note that the estimates in \eqref{e:Sthetamain} are equivalent to sharp lower and upper weighted $L^2$-estimates for the
Fourier extension operator $G \mapsto \widehat{G \mathrm{d}\sigma}$, with a homogeneous weight, and incorporating certain angular regularity given by $\theta$. For certain natural choices of $\theta$, one may simultaneously obtain $\inf_{k} \lambda_k(\theta) > 0$ \emph{and} $\sup_{k} \lambda_k(\theta) < \infty$ so that \eqref{e:Sthetamain} is a \emph{bona fide} equivalence. In fact, choosing $\theta(\varrho) = (1+\varrho)^{\frac{2s-1}{4}}$ we obtain the following.
\begin{corollary} \label{c:FangWangsharp}
Let $d \geq 2$ and $s \in (\frac{1}{2},\frac{d}{2})$. Then
\begin{equation} \label{e:FangWang}
 \inf_{k \in \mathbb{N}_0} \widetilde{\lambda}_k \|G\|_{L^2(\mathbb{S}^{d-1})}^2 \leq \| |\cdot|^{-s} (1-\Lambda)^{\frac{2s-1}{4}} \widehat{G \mathrm{d}\sigma} \|_{L^2(\mathbb{R}^d)}^2 \leq  \sup_{k \in \mathbb{N}_0} \widetilde{\lambda}_k \|G\|_{L^2(\mathbb{S}^{d-1})}^2\,,
\end{equation}
where
\[
\widetilde{\lambda}_k = (2\pi)^d 2^{1-2s} \frac{\Gamma(2s-1)\Gamma(k+\frac{d}{2}-s)}{\Gamma(s)^2\Gamma(k+\frac{d}{2} - 1 + s)} (1+k(k+d-2))^{s - \frac{1}{2}}
\]
and the constants $\inf_{k} \widetilde{\lambda}_k > 0$ and $\sup_{k} \widetilde{\lambda}_k < \infty$ are optimal. 
\end{corollary}
\begin{remarks}
(1) In Theorem 1.1 of \cite{FangWang}, the equivalence of the quantities 
$$
\| |\cdot|^{-s} (1-\Lambda)^{\frac{2s-1}{4}} \widehat{G \mathrm{d}\sigma} \|_{L^2(\mathbb{R}^d)}^2 \qquad \text{and} \qquad \|G\|_{L^2(\mathbb{S}^{d-1})}^2
$$ 
was established, but the optimal constants were not obtained.

(2) The sharp lower and upper bounds $\inf_{k} \widetilde{\lambda}_k > 0$ and $\sup_{k} \widetilde{\lambda}_k < \infty$ in Corollary \ref{c:FangWangsharp}, along with the associated space of extremisers, can be calculated more explicitly, but it is not straightforward and the values of $k \in \mathbb{N}_0$ at which these extrema are attained (if any) depends in a very delicate way on the dimension $d$ and the relative size of $s$ to $d$. When $d=2,3$, we have
\[
\inf_{k \in \mathbb{N}_0} \widetilde{\lambda}_k = \lim_{k \to \infty} \widetilde{\lambda}_k = 2^{1-2s} \frac{\Gamma(2s-1)}{\Gamma(s)^2}
\]
and 
\[
\sup_{k \in \mathbb{N}_0} \widetilde{\lambda}_k = \widetilde{\lambda}_0 = 2^{1-2s} \frac{\Gamma(2s-1)\Gamma(\frac{d}{2} - s)}{\Gamma(s)^2 \Gamma(\frac{d}{2} + s - 1)}\,. 
\]
Also, in such dimensions, there are no extremisers for the lower bound, and constant functions on $\mathbb{S}^{d-1}$ are the only extremisers for the upper bound. For $d \geq 4$ the situation is rather more complicated and since this is not the main focus of this paper, we refer the reader to Theorem 2.1 of \cite{BS2} where this information can easily be extracted.

(3) We highlight one further very distinguished case where $d=4$ and $s=1$. Here, one can show that $\widetilde{\lambda}_k = \frac{1}{2}$ for \emph{all} $k \in \mathbb{N}_0$ and hence
\eqref{e:FangWang} is the rather miraculous identity
\begin{equation*}
\| |\cdot|^{-1} (1-\Lambda)^{\frac{1}{4}} \widehat{G \mathrm{d}\sigma} \|_{L^2(\mathbb{R}^4)}^2 = \tfrac{(2\pi)^d}{2} \|G\|_{L^2(\mathbb{S}^{3})}^2
\end{equation*}
for all $G \in L^2(\mathbb{S}^3)$.
\end{remarks}

\subsection*{Sharp trace estimates with angular regularity}
We recall that 
\[
\mathcal{S}_\theta g= \mathcal{R} \, \overline{\theta}(-\Lambda)  D^{-s}g \qquad \text{and} \qquad \mathcal{S}_\theta^* G = (|\cdot|^{-s} \theta(-\Lambda) \widehat{G \mathrm{d}\sigma})^{\vee}\,.\] 
From Theorem \ref{t:Stheta} we may obtain the following sharp trace estimates which allow the inclusion of angular regularity.
\begin{corollary} \label{c:ourtrace}
Let $d \geq 2$ and $s \in (\frac{1}{2},\frac{d}{2})$. Then, for $\theta$ such that $\sup_{k} \lambda_k(\theta) < \infty$ we have
\begin{equation} \label{e:sharptracegeneral}
\| \mathcal{R} \, \overline{\theta}(-\Lambda) f \|^2_{L^2(\mathbb{S}^{d-1})} \leq \sup_{k \in \mathbb{N}_0} \lambda_k(\theta) \|f\|^2_{\dot{H}^s(\mathbb{R}^d)}
\end{equation}
where the constant is optimal and the space of extremisers is precisely the nonzero elements of $D^{-s}\mathcal{S}_\theta^*(\bigoplus_{k \in \mathbf{K}} \mathcal{H}_k)$.
\end{corollary}
\begin{proof} 
The estimate \eqref{e:sharptracegeneral} is an immediate consequence of the upper bound in \eqref{e:Sthetamain} and duality, so it remains to characterise the space of extremisers. For this, it suffices to show that
\begin{equation} \label{e:dualex}
\| \mathcal{S}_\theta g\|_{L^2(\mathbb{S}^{d-1})}^2 = \sup_{k \in \mathbb{N}_0} \lambda_k(\theta) \|g\|_{L^2(\mathbb{R}^d)}^2
\end{equation}
if and only if $g \in \mathcal{S}_\theta^*(\bigoplus_{k \in \mathbf{K}} \mathcal{H}_k)$. When $\mathbf{K}$ is empty we know that there are no extremisers for the upper bound in \eqref{e:Sthetamain} and it follows that there are no $g \in L^2(\mathbb{R}^d) \setminus \{0\}$ satisfying \eqref{e:dualex}. Thus, it remains to consider the case where $\mathbf{K}$ is nonempty, and we take any $k \in \mathbf{K}$. If $g \in L^2(\mathbb{R}^d) \setminus \{0\}$ is such that \eqref{e:dualex} holds, and if $G \in L^2(\mathbb{S}^{d-1})$ is given by
\[
G = \frac{\mathcal{S}_\theta g}{\| \mathcal{S}_\theta g \|_{L^2(\mathbb{S}^{d-1})}}\,,
\]
then we have 
\begin{align*}
\lambda_k(\theta) \| g\|_{L^2(\mathbb{R}^d)}^2 = \langle g, \mathcal{S}_\theta^* G \rangle_{L^2(\mathbb{R}^{d})}^2 \leq \|g\|_{L^2(\mathbb{R}^{d})}^2 \|\mathcal{S}_\theta^*G\|_{L^2(\mathbb{R}^{d})}^2 \leq   \lambda_k(\theta) \|g\|_{L^2(\mathbb{R}^{d})}^2 \,.
\end{align*}
This means $G$ is an extremiser for the upper bound in \eqref{e:Sthetamain}, hence $G \in \bigoplus_{\ell \in \mathbf{K}} \mathcal{H}_\ell$; moreover, $g$ and $\mathcal{S}_\theta^*G$ must be linearly dependent, which means $g \in \mathcal{S}_\theta^*(\bigoplus_{\ell \in \mathbf{K}} \mathcal{H}_\ell)$, as desired.

Conversely, suppose $g = \mathcal{S}_\theta^*G$ for some $G = \sum_{\ell \in \mathbf{K}} H_\ell G \in \bigoplus_{\ell \in \mathbf{K}} \mathcal{H}_\ell$ so that $\mathcal{S}_\theta g = \lambda_k(\theta) G$ by \eqref{e:Sthetaeigen}. Hence
\[
\| \mathcal{S}_\theta g \|^2_{L^2(\mathbb{S}^{d-1})} = \lambda_k(\theta) \langle \mathcal{S}_\theta g, G \rangle_{L^2(\mathbb{S}^{d-1})} = \lambda_k(\theta) \langle g, \mathcal{S}_\theta^* G \rangle_{L^2(\mathbb{R}^{d})} = \lambda_k(\theta)\|g\|_{L^2(\mathbb{R}^d)}^2 
\]
and $g$ satisfies \eqref{e:dualex}.
\end{proof}
Except for the case where $\theta = \mathbf{1}$, the optimal constant in \eqref{e:sharptracegeneral} is new. Recall that the optimal constant when $\theta = \mathbf{1}$ was  found in \cite{RStrace}, but no information regarding extremisers was given. Thus, the characterisation of extremisers in Corollary \ref{c:ourtrace} is new in all cases.

\section{Sharp trace theorem into $L^p(\mathbb{S}^{d-1})$}

Let $\mathfrak{L}(d,s)$ denote the space of complex-valued functions $G$ on $\mathbb{S}^{d-1}$ given by
\begin{equation*}
G(\omega) = \frac{c}{(1 - x \cdot \omega)^{\frac{d}{2} + s -1}}
\end{equation*}
for some $c \in \mathbb{C} \setminus \{0\}$ and $x \in \mathbb{R}^d$ with $|x| < 1$. This class of functions is precisely the set of extremisers to the sharp Hardy--Littlewood--Sobolev estimate on $\mathbb{S}^{d-1}$, due to Lieb \cite{Lieb},
\begin{equation} \label{e:Lieb}
\int_{\mathbb{S}^{d-1}} \int_{\mathbb{S}^{d-1}} \frac{G(\omega) \overline{G(\varphi)}}{|\omega - \varphi|^{d-2s}} \, \mathrm{d}\sigma(\omega) \mathrm{d}\sigma(\varphi) \leq \mathbf{L}(d,s) \|G\|_{L^q(\mathbb{S}^{d-1})}^2\,,
\end{equation}
where $q = \frac{2(d-1)}{d + 2s - 2}$ and
\[
\mathbf{L}(d,s) = \pi^{\frac{d-2s}{2}} \frac{\Gamma(\frac{2s-1}{2})}{\Gamma(\frac{d}{2}+s-1)} \bigg( \frac{\Gamma(d-1)}{\Gamma(\frac{d-1}{2})} \bigg)^{\frac{2s-1}{d-1}}
\]
is the optimal constant. Clearly $\mathfrak{L}(d,s)$ contains the (nonzero) constant functions on $\mathbb{S}^{d-1}$ by taking $x=0$.
\begin{theorem} \label{t:sharpsharptrace}
Let $d \geq 2$, $s \in (\frac{1}{2},\frac{d}{2})$ and $p = \frac{2(d-1)}{d-2s}$. Then
\begin{equation} \label{e:sharpsharptrace}
\| \mathcal{R} f \|_{L^{p}(\mathbb{S}^{d-1})}^2 \leq 2^{1-2s} \frac{\Gamma(2s-1)\Gamma(\frac{d}{2}-s)}{\Gamma(s)^2\Gamma(\frac{d}{2} - 1 + s)} \bigg( \frac{\Gamma(\frac{d}{2})}{2\pi^{\frac{d}{2}}} \bigg)^{\frac{2s-1}{d-1}}  \| f\|_{\dot{H}^s(\mathbb{R}^d)}^2
\end{equation}
for all $f \in \dot{H}^s(\mathbb{R}^d)$. The constant is optimal and equality holds if and only if 
\[
f = \frac{1}{|\cdot|^{d-2s}} * G \mathrm{d}\sigma
\]
for some $G \in \mathfrak{L}(d,s)$.
\end{theorem}
\begin{proof}
Let $C(d,s)$ be given by
\[
C(d,s) = 2^{1-2s} \frac{\Gamma(2s-1)\Gamma(\frac{d}{2}-s)}{\Gamma(s)^2\Gamma(\frac{d}{2} - 1 + s)} \bigg( \frac{\Gamma(\frac{d}{2})}{2\pi^{\frac{d}{2}}} \bigg)^{\frac{2s-1}{d-1}}\,.
\]
If $q = \frac{2(d-1)}{d + 2s - 2}$ then it follows from \eqref{e:SS*}, \eqref{e:Lieb} and the duplication formula for the Gamma function
$\Gamma(z)\Gamma(z+\tfrac{1}{2}) = 2^{1-2z} \pi^{\frac{1}{2}}\Gamma(2z)$, that
\begin{equation} \label{e:sharpsharptracedual}
\| \mathcal{S}^* G \|_{L^2(\mathbb{R}^d)}^2 = \langle \mathcal{S}\mathcal{S}^* G, G \rangle \leq C(d,s) \|G\|_{L^q(\mathbb{S}^{d-1})}^2\,.
\end{equation}
Here, the constant is optimal and $G$ is an extremiser if and only if $G \in \mathfrak{L}(d,s)$. Since $q' = p$, by duality, we get
\begin{equation} \label{e:almostsharpsharptrace}
\| \mathcal{S}g\|_{L^{p}(\mathbb{S}^{d-1})}^2 \leq C(d,s) \|g\|_{L^2(\mathbb{R}^d)}^2
\end{equation}
and hence \eqref{e:sharpsharptrace}.

It remains to characterise the extremsiers for \eqref{e:sharpsharptrace}. To do this, we follow the idea of the proof of Theorem \ref{t:extremisers}. First, if $g$ is an extremiser for \eqref{e:almostsharpsharptrace} then we define
\[
G = \frac{|\mathcal{S} g|^{p-2}\mathcal{S} g}{\|\mathcal{S} g\|_{L^{p}(\mathbb{S}^{d-1})}^{p-1}}
\]
so that $\|G\|_{L^q(\mathbb{S}^{d-1})} = 1$ and
\begin{align*}
C(d,s) \| g\|_{L^2(\mathbb{R}^d)}^2 & = \| \mathcal{S} g \|^2_{L^{p}(\mathbb{S}^{d-1})}  \\
& = \langle \mathcal{S} g, G \rangle_{L^2(\mathbb{S}^{d-1})}^2 \\
& = \langle g, \mathcal{S}^* G \rangle_{L^2(\mathbb{R}^{d})}^2 \leq \|g\|_{L^2(\mathbb{R}^{d})}^2 \|\mathcal{S}^*G\|_{L^2(\mathbb{R}^{d})}^2 \leq   C(d,s)  \| g\|_{L^2(\mathbb{R}^d)}^2 \,.
\end{align*}
Equality throughout implies firstly that $G$ is an extremiser for \eqref{e:sharpsharptracedual} and hence $G \in \mathfrak{L}(d,s)$. Secondly, equality at the application of Cauchy--Schwarz forces $g$ and $\mathcal{S}^* G$ to be linearly dependent. So, if $g$ is an extremiser for \eqref{e:almostsharpsharptrace}, then $g$ belongs to the image of $\mathfrak{L}(d,s)$ under $\mathcal{S}^* $. 

Now we show that such $g$ are indeed extremisers for \eqref{e:almostsharpsharptrace}. So, suppose $g = \mathcal{S}^* G$ for some $G \in \mathfrak{L}(d,s)$. Since $G$ is an extremiser for \eqref{e:Lieb}, in particular, it satisfies the corresponding Euler--Lagrange equation, which one can check is
\begin{equation} \label{e:ELLieb}
\int_{\mathbb{S}^{d-1}} \frac{G(\varphi)}{|\omega - \varphi|^{d-2s}} \, \mathrm{d}\sigma(\varphi) = \lambda(G) |G(\omega)|^{q-2}G(\omega)\,,
\end{equation}
where
\[
\lambda(G) = \frac{1}{\|G\|_{L^q(\mathbb{S}^{d-1})}^q} \int_{\mathbb{S}^{d-1}} \int_{\mathbb{S}^{d-1}} \frac{G(\omega) \overline{G(\varphi)}}{|\omega - \varphi|^{d-2s}} \, \mathrm{d}\sigma(\omega)\mathrm{d}\sigma(\varphi)\,.
\]
Using that $G$ is an extremiser for \eqref{e:Lieb}, we obtain
\[
\lambda(G) = \mathbf{L}(d,s) \|G\|_{L^q(\mathbb{S}^{d-1})}^{2-q}
\]
and therefore, by \eqref{e:SS*}, we have
\begin{equation} \label{e:geneigen}
\mathcal{S} \mathcal{S}^* G(\omega) = C(d,s) \|G\|_{L^q(\mathbb{S}^{d-1})}^{2-q} |G(\omega)|^{q-2}G(\omega)\,.
\end{equation}
Using \eqref{e:geneigen} we have
\begin{align*}
\|\mathcal{S}g\|_{L^{p}(\mathbb{S}^{d-1})}^2 = \|\mathcal{S}\mathcal{S}^*G\|_{L^{p}(\mathbb{S}^{d-1})}^2 = C(d,s)^2 \|G\|_{L^q(\mathbb{S}^{d-1})}^{2}
\end{align*}
and using that $G$ is an extremiser for \eqref{e:sharpsharptracedual}, it follows that $g$ is an extremiser for \eqref{e:almostsharpsharptrace}, as desired.
\end{proof}

The exponent $p = \frac{2(d-1)}{d-2s}$ in Theorem \ref{t:sharpsharptrace} cannot be improved in the sense that $\mathcal{R}$ is not a bounded operator from $\dot{H}^s(\mathbb{R}^d)$ to $L^p(\mathbb{S}^{d-1})$ for any $p > \frac{2(d-1)}{d-2s}$. This can be seen, for example, by showing that $q \geq \frac{2(d-1)}{d + 2s - 2}$ is necessary for \eqref{e:sharpsharptracedual} by testing on the so-called Knapp example (where $G$ is a characteristic function of a small cap on the sphere).

Lieb's sharp Hardy--Littlewood--Sobolev inequality \eqref{e:Lieb} and characterisation of extremisers (along with the well-known characterisation of extremisers for H\"older's inequality) is more than enough to prove Theorem \ref{t:extremisers}, since Theorem \ref{t:sharpsharptrace} is a stronger result via an application of H\"older's inequality. Our argument to prove Theorem \ref{t:extremisers} in Section \ref{section:theta1} is based on Theorem \ref{t:S1} which is perhaps more natural if one wishes to consider the classical trace theorem $\dot{H}^s(\mathbb{R}^d) \to L^2(\mathbb{S}^{d-1})$, since it is a slightly more direct approach and readily permits the generalisations in Section \ref{section:angular}. 

Interestingly, a key aspect of the proof of Theorem \ref{t:S1} and the proof of Lieb's inequality \eqref{e:Lieb} given in \cite{FrankLieb} is that integral operators on the sphere with kernels $K(\omega \cdot \varphi)$ diagonalise with respect to the spherical harmonic decomposition of $L^2(\mathbb{S}^{d-1})$, and the eigenvalues are explicitly computable via the Funk--Hecke formula. The argument in \cite{FrankLieb} leading to \eqref{e:Lieb} is, unsurprisingly, more involved than the argument required to prove Theorem \ref{t:S1}, which follows almost immediately from the Funk--Hecke formula. Also, such a diagonalisation property and the Funk--Hecke formula were important ingredients in fundamental work of Beckner \cite{Beckner} and the very recent and elegant argument of Foschi \cite{Foschi} where the optimal constant and extremisers were identified for the Stein--Tomas extension estimate $L^2(\mathbb{S}^2) \to L^{4}(\mathbb{R}^3)$ for the Fourier extension operator $G \mapsto \widehat{G \mathrm{d}\sigma}$ associated to the sphere in $\mathbb{R}^3$.

\end{document}